\documentclass[11pt]{article}

\usepackage[margin=1in]{geometry}
\usepackage{amsmath,amssymb,amsthm,mathtools}
\usepackage{enumitem}
\usepackage[hidelinks]{hyperref}

\newtheorem{theorem}{Theorem}
\newtheorem{lemma}{Lemma}
\newtheorem{proposition}{Proposition}
\newtheorem{remark}{Remark}
\newtheorem{example}{Example}
\newtheorem{conjecture}{Conjecture}

\newcommand{\cB}{\mathcal{B}}
\newcommand{\cF}{\mathcal{F}}
\newcommand{\cH}{\mathcal{H}}
\newcommand{\cL}{\mathcal{L}}
\newcommand{\cW}{\mathcal{W}}
\newcommand{\one}{\mathbf{1}}
\newcommand{\st}{\ast}

\title{On balanced subfamilies of maximum complement-free families in the middle layer of the Boolean lattice}
\author{Sa\'ul A. Blanco\footnote{Department of Computer Science, Indiana University, Bloomington. \\ Email: \texttt{sblancor@iu.edu}}}
\date{October 15, 2026}

\begin{document}
\maketitle

\begin{abstract}
We study balanced subfamilies of the middle layer $\binom{[2n]}{n}$ of the Boolean lattice $2^{[2n]}$. A family $\mathcal{F}\subseteq\binom{[2n]}{n}$ is said to be balanced if every element in $[2n]$ appears in the same number of members of $\mathcal{F}$. A balanced subfamily of size 2 is exactly a complementary pair $\{A,[2n]\setminus A\}$, and therefore a family with no balanced subfamily of size $2$ has at most $\frac{1}{2}\binom{2n}{n}$ members. We show that for every $k\geq 1$ and all sufficiently large $n$, this maximum size is compatible with delaying the smallest size of a balanced subfamily until $2k+2$. More precisely, there exists a family $\mathcal{F}\subseteq\binom{[2n]}{n}$ of size $\frac{1}{2}\binom{2n}{n}$ with no balanced subfamilies of sizes $2,4,\ldots,2k$, but with a balanced subfamily of size $2k+2$. The proof is constructive and uses a general lifting from simple games to self-dual selectors in the middle layer. This provides an efficient constructive proof of a recent conjecture of Moss and Pedersen. 
\end{abstract}

\noindent\textbf{Keywords:} Balanced families, simple games, trade robustness, complement-free families, middle layer, Boolean lattice.

\section{Preliminaries}

Given a finite set $X$, the set of subsets of $X$ of cardinality $r$ is denoted by $\binom{X}{r}$. If one considers the lattice $(2^X,\subseteq)$ of subsets of $X$ ordered by inclusion, then the sets in $\binom{X}{r}$ form the $r$-th layer of this lattice. If $|X|=2r$, then $\binom{X}{r}$ is referred to as the \emph{middle layer}. For a positive integer $m$, let $[m]=\{1,2,\dots,m\}$.  A family
\[
  \cB\subseteq \binom{[2n]}{n}
\]
is called \emph{balanced} if every element in $[2n]$ belongs to the same number of members of $\cB$.  Note that a balanced family of size $2$ is simply a complementary pair $\{A,[2n]\setminus A\}$. Therefore any family in $\binom{[2n]}{n}$ with no balanced subfamily of size $2$ has size at most $\frac{1}{2}\binom{2n}{n}$. Suppose $|\mathcal B|=s$. Since each set in $\mathcal{B}$ has cardinality $n$, the sum of the cardinalities of the members in $\mathcal B$ is $sn$. If every element of $[2n]$ belongs to exactly $r$ members of $\mathcal{B}$, then this sum is also $2nr$. As a consequence, $s=2r$. In particular, every balanced subfamily has even cardinality, and a balanced subfamily of size $2r$ contains each element of $[2n]$ in exactly $r$ of its members. 

Let $\one_A$ denote the \emph{characteristic vector} of $A\subseteq [2n]$. Then, a sequence
\[
  A_1,\dots,A_{2r}\in \binom{[2n]}{n}
\]
is balanced if and only if $\one_{A_1}+\cdots+\one_{A_{2r}}=\underbrace{(r,r,\dots,r)}_{2n}$.

In a recent pre-print, Moss and Pedersen~\cite{MossPedersen} conjectured the following statement.

\begin{conjecture}\label{con:Moss}
Fix an integer $k\geq 1$. For all sufficiently large $n$, there exists a family $\cF\subseteq \binom{[2n]}{n}$ such that:

\begin{enumerate}
    \item[(i)] $\cF$ has maximum possible size subject to containing no complementary pair; that is,
    \[
    |\cF|=\frac{1}{2}\binom{2n}{n}.
    \]
    \item[(ii)] $\cF$ contains no balanced subfamily of size $2,4,\ldots,2k$.
    \item[(iii)] $\cF$ contains a balanced subfamily of size $2k+2$.
\end{enumerate}
\end{conjecture}

In this note, we provide an efficient constructive proof of Conjecture~\ref{con:Moss}.

In a subsequent pre-print released after an earlier version of this paper, Moss and Pedersen~\cite{MossPedersen2026b} gave an alternative construction of the desired balanced families. Their approach is self-contained, relying on linear algebra and geometric arguments rather than results from the theory of simple games. They establish that for each $k\geq1$, their construction produces such a family in the middle layer $\binom{[2n]}{n}$ with $n=\frac12\binom{2k+2}{k+1}$. In contrast, our construction gives the linear bound $N(k)\leq 2k+4$ for $k\geq3$, and the padding argument establishes the result for every $n\geq 2k+4$.

\subsection{Simple games, trades, and robustness}

A \emph{simple game} is a pair $G=(P,\cW)$, where $P$ is a finite set of players and $\cW\subseteq 2^P$ is the set of winning coalitions.  The losing coalitions are the complement $\cL=2^P\setminus \cW$. $G$ is said to be
\emph{monotone} if $X\in\mathcal W$ and $X\subseteq Y\subseteq P$ implies that $Y\in\mathcal W$.\footnote{Taylor and Zwicker~\cite{TaylorZwicker1995} remark that in much of game theory, the definition of simple game includes this monotonicity condition; for example this is the convention used by Gvozdeva and Slinko in~\cite{GvozdevaSlinko2011}. We follow Taylor and Zwicker and do not require monotonicity in the definition since this is not needed in our proofs.}

For $t\geq1$, a $t$-\emph{trade} from winning coalitions to losing coalitions is a pair of
sequences of sets
\[
  X_1,\dots,X_t\in \cW\qquad\text{ and }\qquad Y_1,\dots,Y_t\in \cL,
\]
such that every player occurs the same number of times on both sides.  In characteristic-vector notation, this means
\[
  \one_{X_1}+\cdots+\one_{X_t}=\one_{Y_1}+\cdots+\one_{Y_t}.
\]

In general, the coalitions occurring in a trade need not be distinct. 

Following~\cite{GvozdevaSlinko2011}, a simple game is called $t$-\emph{trade robust} if no such $j$-trade exists for any $j\leq t$. Note that Taylor and Zwicker~\cite{TaylorZwicker1995} use the term \textit{assumable} and reserve the term ``$t$-trade robust" to mean that no $t$-trade exists. 

\begin{proposition}[cf. Taylor and Zwicker~\cite{TaylorZwicker1995}, Theorem 3.3 and Theorem 3.4, and Gvozdeva and Slinko~\cite{GvozdevaSlinko2011}, Theorem 5]\label{prop:constructions} The following statements are true.

\begin{enumerate}
    \item[(i)] For each integer $t\geq3$, there exists a game with $2t+3$ players, that is $t$-trade robust but not $(t+1)$-trade robust. This is Gvozdeva and Slinko~\cite[Corollary 3]{GvozdevaSlinko2011}.
    \item[(ii)] For each integer $t\geq 2$, there exists a game with $(t+1)^2$ players, that is $t$-trade robust but not $(t+1)$-trade robust. This follows from Taylor and Zwicker~\cite[Theorem 3.3 and Theorem 3.4]{TaylorZwicker1995} (see also~\cite[Theorem 2]{GvozdevaSlinko2011}). Indeed, for $t\geq2$, they construct a simple game $G_t$ having no $j$-trade if and only if $j$ is not a multiple of $t+1$. Since $j$ is not a multiple of $t+1$ for all $1\leq j\leq t$, and since $t+1$ is a multiple of $t+1$, $G_t$ is $t$-trade robust but not $(t+1)$-trade robust.\footnote{Taylor and Zwicker use the term $j$-assumable for what we define as $j$-trade robust}. 
\end{enumerate}

Furthermore, in each of these constructions, the coalitions obtained on each side of the trade are distinct. This will become important later.
\end{proposition}

As a consequence, we obtain the following result. 

\begin{theorem}\label{thm:one}
For every integer $t\geq 1$, there exists a simple game $G$ on $p(t)$ players that is $t$-trade robust but not $(t+1)$-trade robust. Moreover, $G$ has a $(t+1)$-trade with distinct coalitions on each side. 
\end{theorem}

Notice that if a simple game $G$ is $t$-trade robust but not $(t+1)$-trade robust, then the smallest $j$ for which a $j$-trade exists is $j=t+1$. 

\begin{proof}
    Notice that every simple game is $1$-trade robust, since otherwise there would be a $1$-trade. By definition, a 1-trade would imply that there is one coalition that is both winning and losing, as every player occurs the same number of times on both sides of the 1-trade, which is impossible. Moreover, if $t=1$, we can consider the simple game on 4 players $\{1,2,3,4\}$ where a coalition is winning if and only if it contains $\{1,2\}$ or $\{3,4\}$. Then $\{1,2\},\{3,4\}\longleftrightarrow \{1,4\},\{2,3\}$ is a 2-trade. This establishes the theorem for the case $t=1$.

    If $t\geq2$, this follows from Proposition~\ref{prop:constructions}.
\end{proof}
\begin{remark}
Gvozdeva and Slinko~\cite[Corollary 3]{GvozdevaSlinko2011} give Theorem~\ref{thm:one} for $t\geq3$. It is also relevant to mention that the construction in~\cite{GvozdevaSlinko2011} is more efficient than that in Taylor and Zwicker~\cite{TaylorZwicker1995}. More specifically, the latter uses $(t+1)^2$ players for $t\geq2$, and the former uses $2t+3$ players for $t\geq3$. 
\end{remark}

\section{A self-dual selector}\label{sec:self-dual}

Let $N$ be a finite set. A sequence $U_1,\dots,U_{2j}$ of subsets of $N$ is called \emph{$j$-balanced} if every element of $N$ belongs to exactly $j$ of the sets $U_1,\dots,U_{2j}$. Equivalently,
\[
\sum_{i=1}^{2j} \one_{U_i}=\underbrace{(j,j,\dots,j)}_{|N|}.
\]

Let $X$ be a finite set, and let $\mathcal A\subseteq 2^X$ be closed under taking complements; that is, if $A\in\mathcal A$, then $X\setminus A\in\mathcal A$. A family $\mathcal{S}\subseteq\mathcal A$ is called a \emph{self-dual selector} of $\mathcal{A}$ if, for every $A\in\mathcal{A}$, exactly one of $A$ and $X\setminus A$ belongs to $\mathcal{S}$. Equivalently, $A\in\mathcal{S} \iff X\setminus A\notin\mathcal S$. In other words, $\mathcal{S}$ is a maximum complement-free family of subsets in $\mathcal{A}$.

In particular, since $\binom{[2n]}{n}$ is partitioned into complementary pairs $\left\{\{A,A^c\}:A\in\binom{[2n]}{n}\right \}$, every self-dual selector $\mathcal S\subseteq\binom{[2n]}{n}$ has cardinality $|\mathcal S|=\frac12\binom{2n}{n}.$

Let $G=(P,\cW)$ be any simple game, and let $\cL=2^P\setminus \cW$ be its set of losing coalitions.  Let us introduce one new point, denoted $\st$, and define the family $\cH\subseteq 2^{P\cup\{\st\}}$ by
\[
  \cH=\{X\cup\{\st\}:X\in \cW\}\cup\{P\setminus Y:Y\in \cL\}.
\]

More precisely, for every $U\subseteq P\cup\{\st\}$, exactly one of $U$ and its
complement belongs to $\cH$. Indeed, every complementary pair has the form
\[
  X\cup\{\st\}\qquad\text{ and }\qquad P\setminus X,
\]
for some $X\subseteq P$.  The first set belongs to $\cH$ exactly when $X\in \cW$, and the second belongs to $\cH$ exactly when $X\in \cL$. Since exactly one of these is true, $\cH$ chooses exactly one member from each complementary pair.

The key point is that $j$-balanced sequences of length $2j$ in $\mathcal{H}$ correspond to $j$-trades in $G$. Moreover, when the coalitions on each side of the trade are distinct, the corresponding sequence is a balanced subfamily of size $2j$.

\begin{lemma}\label{lem:balanced-trade}
Let $G=(P,\cW)$ be a simple game, and let $\cH$ be the self-dual selector constructed above.  Then $j$-trades in $G$ correspond to $j$-balanced sequences of length $2j$ in $\mathcal{H}$. Moreover, $j$-trades whose coalitions on each side are distinct correspond to $j$-balanced subfamilies of $\cH$ of size $2j$.
\end{lemma}

\begin{proof}
Suppose $U_1,\dots,U_{2j}\in \cH$ is $j$-balanced on $P\cup\{\st\}$.  Since the special point $\st$ must occur
exactly $j$ times, after reordering we may write
\[
  U_i=X_i\cup\{\st\},\text{ with } 1\leq i\leq j,\qquad\text{ and }\qquad U_{j+i}=P\setminus Y_i,\text{ with } 1\leq i\leq j,
\]
where each $X_i$ is winning in $G$ and each $Y_i$ is losing in $G$, respectively.

Now fix a player $p\in P$.  The $j$-balance condition gives that $p$ occurs exactly $j$ times
among the $2j$ sets $U_1,\dots,U_{2j}$.  Hence,
\[
  |\{i:p\in X_i\}|
  +
  |\{i:p\in P\setminus Y_i\}|
  =j.
\]
On the other hand, $|\{i:p\in P\setminus Y_i\}|=j-|\{i:p\in Y_i\}|$, and therefore $|\{i:p\in X_i\}|=|\{i:p\in Y_i\}|$.

Since this holds for every $p\in P$, it follows that $\one_{X_1}+\cdots+\one_{X_j}=\one_{Y_1}+\cdots+\one_{Y_j},$ and so $X_1,\dots,X_j \longleftrightarrow Y_1,\dots,Y_j$ is a $j$-trade in $G$.

Conversely, suppose we are given a $j$-trade $X_1,\dots,X_j \longleftrightarrow Y_1,\dots,Y_j$, where the $X_i$ are winning and the $Y_i$ are losing, respectively.  Then the $2j$ sets
\[
  X_1\cup\{\st\},\dots,X_j\cup\{\st\},
  \qquad\text{ and }\qquad
  P\setminus Y_1,\dots,P\setminus Y_j
\]
belong to $\cH$.  The point $\st$ appears exactly $j$ times.  Indeed, for a player $p\in P$, if $p$ occurs $a_p$ times among the $X_i$, then by the trade condition it also occurs $a_p$ times among the $Y_i$.  Hence $p$ occurs $a_p+(j-a_p)=j$ times among the shown $2j$ members $X_1\cup\{\st\},\dots,X_j\cup\{\st\}, P\setminus Y_1,\dots,P\setminus Y_j$ of $\cH$.  Therefore the sequence $X_1\cup\{\st\},\dots,X_j\cup\{\st\}, P\setminus Y_1,\dots,P\setminus Y_j$ is $j$-balanced.

It is worth noting that if the coalitions on each side of the trade are distinct, then the resulting $2j$ sets are distinct. Indeed, the sets $X_i\cup{\st}$ are distinct because the $X_i$ are distinct, and the sets $P\setminus Y_i$ are distinct because the $Y_i$ are distinct. Furthermore, a set of the first type cannot equal a set of the second type, since the former contains $\st$ while the latter does not. Therefore, in this case the resulting sequence is a $j$-balanced subfamily of $\mathcal H$ of size $2j$.
\end{proof}
\begin{remark}
    While the choice of $\st$ might seem artificial, it serves an important purpose: it records the winning side of the trade. Indeed, in any $j$-balanced sequence of length $2j$ in $\mathcal H$, the point $\st$ must occur exactly $j$ times, and hence exactly $j$ of the sets arise from winning coalitions. This is what allows us to recover a $j$-trade from a $j$-balanced sequence. The use of $\st$ is inspired by the idea of constant-sum extensions mentioned in~\cite{TaylorZwicker1995} (see also~\cite{Einy1985,VonNeumannMorgenstern1944}). 
\end{remark}

\section{Proof of Conjecture~\ref{con:Moss}}
We now prove the main theorem. 

\begin{theorem}\label{thm:main}
For every integer $k\geq1$, there exists an integer $n$ so that there exists a self-dual selector $\cF\subseteq\binom{[2n]}{n}$ whose smallest balanced subfamily has size $2k+2$. In particular, $|\mathcal{F}|=\frac{1}{2}\binom{2n}{n}$.
\end{theorem}

The proof of Theorem~\ref{thm:main} is a construction that turns a simple game that is $k$-trade robust but not $(k+1)$-trade robust using distinct coalitions on each side into a $(2k+2)$-balanced subfamily. The existence of such games is guaranteed by Theorem~\ref{thm:one}. The self-dual selector from Section~\ref{sec:self-dual} and Lemma~\ref{lem:balanced-trade} provide the construction of said games into a balanced subfamily with the right conditions. 

\begin{proof}
    Given $k\geq1$, let $G=(P,\mathcal{W})$ be a simple game with set of players $P$ and set of winning coalitions $\mathcal{W}$. Moreover, by Theorem~\ref{thm:one}, we may assume that $G$ is $k$-trade robust and that $G$ has a $(k+1)$-trade with distinct coalitions on either side robust. Thus Lemma~\ref{lem:balanced-trade} gives us a correspondence of this $(k+1)$-trade to a $(k+1)$-balanced family of size $2(k+1)$. What remains is finding an $n$ so that a middle layer of the lattice $(2^{[2n]},\subseteq)$ contains an appropriate lifting of this $(k+1)$-balanced family to $n$-sets.
    Let $S=P\cup\{\st\}$ and let $S'=\{s':s\in S\}$ and take $N=S\sqcup S'$. In other words, $S'$ is a copy of $S$ and $N$ is the disjoint union of $S$ and $S'$. In particular, $|N|=2|S|$. Let
    \[
    X_1,X_2,\ldots,X_{k+1}\longleftrightarrow Y_1,Y_2,\ldots, Y_{k+1}
    \] be a $(k+1)$-trade with \textit{distinct} coalitions, as provided by Theorem~\ref{thm:one}. For $U\subseteq S$, consider $U':=U\cup\{s':s\in S\setminus U\}$. In addition, let
    \[
    U_i=X_i\cup\{\st\} \text{ for } 1\leq i\leq k+1\qquad\text{ and }\qquad  U_{k+1+i}=P\setminus Y_i\text{ for }1\leq i\leq k+1.
    \] The proof of Lemma~\ref{lem:balanced-trade} gives that $U_1,\ldots, U_{2k+2}$ is a $(k+1)$-balanced family and that 
    \[
    \mathcal{H}=\{X\cup\{\st\}:X\in\mathcal{W}\}\cup\{P\setminus Y:Y\in\mathcal{L}\},
    \] is a self-dual selector on $(2^S,\subseteq)$. Additionally, notice that $s'\in U'_i$ if and only if $s\not\in U_i$. Recall that $U_1,\ldots, U_{2k+2}$ is a $(k+1)$-balanced family, and so every $s\in S$ is in exactly $k+1$ of the sets in $\{U_1,\ldots, U_{2k+2}\}$. Thus every $s\in S$ is also missing from exactly $k+1$ of the sets in $\{U_1,\ldots, U_{2k+2}\}$, or equivalently every element $s'\in S'$ is in exactly $k+1$ of the sets in $\{U'_1,\ldots, U'_{2k+2}\}$. Therefore, $U'_1,U'_2,\ldots, U'_{2k+2}$ is a $(k+1)$-balanced family. Moreover, since $|\{s':s\in S\setminus U_i\}|=|S|-|U_i|$, it follows that $|U'_i|=|U_i|+(|S|-|U_i|)=|S|$. Therefore, $U'_1,\ldots, U'_{2k+2}$ is a $(k+1)$-balanced family of size $2k+2$, each set containing $|S|$ elements.
    
    Consider the family
    \[
    \mathcal{F}'=\left\{A\in\binom{N}{n}:A\cap S\in\mathcal{H}\right\},
    \]where $n=|S|$ and so $2n=|N|$. It is easy to see that $\mathcal{F}'$ is a self-dual selector of $n$-sets in $(2^N,\subseteq)$. Indeed, take $A\in \binom{N}{n}$ and notice that $(N\setminus A)\cap S=S\setminus (A\cap S)$. Since $\mathcal{H}$ is a self-dual selector of $(2^S,\subseteq)$, then either $A\cap S$ or $S\setminus (A\cap S)$ is in $\mathcal{H}$. Therefore, either $A$ or $N\setminus A$ is in $\mathcal{F}'$.

    Let $\phi:N\to[2n]$ be a bijection. Then the desired family $\mathcal{F}\subseteq\binom{[2n]}{n}$ in the middle layer of $(2^{[2n]},\subseteq)$ is $\mathcal{F}=\phi(\mathcal{F}')$, with balanced subfamily of size $2k+2$ being $\{\phi(U'_i):1\leq i\leq 2k+2\}$. 

    It remains to show that $\mathcal{F'}$ does not have a balanced subfamily of size $2j$ with $1\leq j\leq k$. For the sake of contradiction, suppose that $A_1,\ldots,A_{2j}\in\mathcal{F}'$ is balanced. Then, 
    \[
    A_1\cap S, A_2\cap S,\ldots, A_{2j}\cap S
    \] would form a $j$-balanced sequence in $\mathcal{H}$. Thus, by Lemma~\ref{lem:balanced-trade}, this would produce a $j$-trade in $G$. This contradicts the fact that $G$ is $k$-trade robust. 
\end{proof}

We will show that if $n$ is the integer from Theorem~\ref{thm:main}, then, for a fixed $k$, we can extend the result to all integers $m\geq n$. The idea is to construct an appropriate padding that will allow us to perform such an extension. We formalize this with the following lemma. 

The following lemma will be used as part of our construction. 

\begin{lemma}[Padding Lemma]\label{lem:padding}
Suppose that $G=(P,\mathcal{W})$ is $k$-trade robust and has a $(k+1)$-trade with distinct coalitions on each side. Suppose further that the construction of Theorem~\ref{thm:main} applied to the $(k+1)$-trade of distinct coalitions on each side
\[
X_1,\ldots,X_{k+1}\longleftrightarrow Y_1,\ldots,Y_{k+1},
\]
 produces a family $\mathcal{F}\subseteq\binom{[2n]}{n}$ with $|\mathcal{F}|=\frac{1}{2}\binom{2n}{n}$. Then the construction can be extended to produce a family of size $\frac{1}{2}\binom{2n+2}{n+1}$ whose smallest balanced subfamily of $(n+1)$-sets has size $2k+2$.
\end{lemma}
\begin{proof}
The construction from Theorem~\ref{thm:main} constructs the $(k+1)$-balanced subfamily of size $2k+2$ as follows:
\begin{description}
    \item[Set] $S=P\cup\{\st\}$ and consider the self-dual selector on $2^S$ given by
    \[
    \mathcal{H}=\{X\cup\{\st\}:X\in\mathcal{W}\}\cup\{P\setminus Y:Y\in\mathcal{L}\},
    \] with $\mathcal{L}=2^P\setminus\mathcal{W}$.
    \item[Then] the sequence of sets $U_i=X_i\cup\{\st\}$, with $1\leq i\leq k+1$, and $U_{k+1+i}=P\setminus Y_i$, with $1\leq i\leq k+1$ is $(k+1)$-balanced of size $2k+2$.
    \item[Set] $N=S\sqcup S'$ with $S'=\{s':s\in S\}$ being a disjoint copy of $S$. 
    \item[Set] For $U\subseteq S$, let $U'=U\cup\{s':s\in S\setminus U\}$. Then the sequence $U'_1,\ldots, U'_{2k+2}$ is of size $2k+2$, $(k+1)$-balanced, and each set has size $|S|$. This family is contained in the self-dual selector
    \[
    \mathcal{F}'=\left\{A\in\binom{N}{n}:A\cap S\in\mathcal{H}\right\},
    \] with $n=|S|$. After relabeling, $\mathcal{F}'$ becomes the family of size $\frac{1}{2}\binom{2n}{n}$ whose smallest balanced subfamily has size $2k+2$ and comes from the relabeling of $U'_1,\ldots,U'_{2k+2}$.
\end{description}

To extend the construction, consider two new points $a,b$, and let 
\[
U''_i=U'_i\cup\{a\},\text{ with }1\leq i\leq k+1\qquad\text{ and }\qquad U''_{k+1+i}=U'_{k+1+i}\cup\{b\},\text{ with }1\leq i\leq k+1.
\] Notice that exactly $k+1$ of the $U''_1,\ldots,U''_{2k+2}$ contain $a$ and exactly $k+1$ of them contain $b$. Therefore, since $U'_1,\ldots,U'_{2k+2}$ is $(k+1)$-balanced, so is $U''_1,\ldots,U''_{2k+2}$. Further, each $U''_i$ has size $|U''_i|=|U'_i|+1=n+1$ for $1\leq i\leq 2k+2$. Additionally, consider
\[
\mathcal{F}''=\left\{A\in\binom{N\sqcup\{a,b\}}{n+1}:A\cap S\in\mathcal{H}\right\}.
\] Notice that $|\mathcal{F}''|=\frac{1}{2}\binom{2n+2}{n+1}$. Indeed, since 
\[
((N\sqcup\{a,b\})\setminus A)\cap S=((S\sqcup S'\sqcup\{a,b\})\setminus A)\cap S=S\setminus (A\cap S)
\] and $\mathcal{H}$ is a self-dual selector of $S$, either $A\cap S\in\mathcal{H}$ or $S\setminus (A\cap S)\in\mathcal{H}$. After relabeling, $\mathcal{F}''$ becomes a self-dual selector of $(n+1)$-sets contained in $[2n+2]$ that contains the balanced subfamily obtained from relabeling $U''_1,\ldots,U''_{2k+2}$, as desired. 

It remains to argue that the addition of $a,b$ does not produce a smaller balanced subfamily. As before, if $A_1,\ldots, A_{2j}$ is a $j$-balanced sequence in $\mathcal{F}''$, with $1\leq j\leq k$, then $A_1\cap S,\ldots, A_{2j}\cap S$ would correspond to a $j$-balanced sequence in $\mathcal{H}$, which would then imply the existence of a $j$-trade in the game. This contradicts $k$-trade robustness. 
\end{proof}

Combining Theorem~\ref{thm:main} and the Padding Lemma, we obtain the proof of Conjecture~\ref{con:Moss}.

\begin{theorem}\label{thm:main2}
For every integer $k\geq1$ and for all sufficiently large $n$, there exists a self-dual selector $\cF\subseteq\binom{[2n]}{n}$ whose smallest balanced subfamily has size $2k+2$. In particular, $|\mathcal{F}|=\frac{1}{2}\binom{2n}{n}$.
\end{theorem}

We include an example that helps understand the construction from Theorem~\ref{thm:main}.

\begin{example}
    Let us illustrate the construction with a toy example. Consider the simple game with set of players $P=\{1,2\}$ and set of winning coalitions $\mathcal{W}=\{\{1\},\{2\}\}$. The set $\mathcal{L}$ of losing coalitions is therefore $\{\varnothing,\{1,2\}\}$. From this game we will construct a family with a balanced subfamily of size $2k+2=4$ that does not have a balanced subfamily of size $2$. 

    Let $S=\{1,2,\st\}$ and therefore $S'=\{1',2',\st'\}, N=S\sqcup S'$. Consider 
    \[
    X_1=\{1\},\qquad X_2=\{2\}, \qquad Y_1=\varnothing, \qquad\text{ and }\qquad Y_2=\{1,2\}.
    \]
    
    The self-dual selector on $(2^S,\subseteq)$ constructed in the proof of Theorem~\ref{thm:main} is 
    \[
    X_1\cup\{\st\}=\{1,\st\},\qquad X_2\cup\{\st\}=\{2,\st\}, \qquad P\setminus Y_1=\{1,2\},\qquad\text{ and}\qquad P\setminus Y_2=\varnothing.
    \] In this example, $\mathcal{H}=\{\{1,\st\},\{2,\st\},\{1,2\},\varnothing\}$. The corresponding subfamily of size 4 is $\{U'_1,U'_2,U'_3,U'_4\}$ where
    \[
    U_1'=\{1,\st,2'\},\qquad U'_2=\{2,\st,1'\},\qquad U'_3=\{1,2,\st'\}, \qquad\text{ and }\qquad U'_4=\{1',2',\st'\}.
    \] Notice that each element from $N$ appears twice in the above sets $U'_i$ with $1\leq i\leq 4$. The maximum family is
    \begin{align*}
    \mathcal {F}'&=\{A\in\binom{N}{|S|}:A\cap S\in\mathcal{H}\}\\
    &=\{\{1,\st,1'\},\{1,\st,2'\},\{1,\st,\st'\},\{2,\st,1'\},\{2,\st,2'\},\{2,\st,\st'\},\{1,2,1'\},\{1,2,2'\},\{1,2,\st'\},\{1',2',\st'\}\}
    \end{align*} Taking $\phi:N\to[2n]$ given by 
  \[  
1 \mathrel{\overset{\phi}{\mapsto}} 1,
\qquad 2 \mathrel{\overset{\phi}{\mapsto}} 2,
\qquad \st \mathrel{\overset{\phi}{\mapsto}} 3,
\qquad 1' \mathrel{\overset{\phi}{\mapsto}} 4,
\qquad 2' \mathrel{\overset{\phi}{\mapsto}} 5,
\qquad \text{ and }\qquad\st' \mathrel{\overset{\phi}{\mapsto}} 6
\]
     yields the desired family $\mathcal{F}\subseteq\binom{[6]}{3}$, namely
    \begin{align*}
     \mathcal{F}=   \{\{1,3,4\},\{1,3,5\},\{1,3,6\},\{2,3,4\},\{2,3,5\},\{2,3,6\},\{1,2,4\},\{1,2,5\},\{1,2,6\},\{4,5,6\}\}
    \end{align*} that does not have a balanced subfamily of size $2$, but does contain a subfamily of size $4$ formed by the sets
    \[
    \{1,3,5\},\qquad\{2,3,4\},\qquad\{1,2,6\},\qquad\text{ and }\qquad \{4,5,6\}.
    \]
\end{example}

\section{Complexity considerations}
For a family $\mathcal{F}\subseteq\binom{[2n]}{n}$, we let $m(\mathcal{F})$ denote the size of the smallest balanced subfamily contained in $\mathcal{F}$. If $\mathcal{F}$ does not have a balanced subfamily, then $m(\mathcal{F})=\infty$.

Let us define
\begin{equation*}
N(k)=\min\left\{n:\text{ there exists }\mathcal{F}\subseteq\binom{[2n]}{n}\text{ with }|\cF|=\frac{1}{2}\binom{2n}{n}\text{ and }m(\mathcal{F})=2k+2\right\},
\end{equation*} and
\begin{equation*}
K(n)=\max\left\{k:\text{ there exists }\mathcal{F}\subseteq\binom{[2n]}{n}\text{ with }|\cF|=\frac{1}{2}\binom{2n}{n}\text{ and }m(\mathcal{F})=2k+2\right\}.
\end{equation*}

We adopt the convention that $\max\varnothing=0$.

Intuitively, given $k$, $N(k)$ is the minimum number $n$ that is needed for a maximum complement-free family $\mathcal{F}\subseteq\binom{[2n]}{n}$ to exist such that $m(\mathcal{F})=2k+2$. Similarly, given $n$, $K(n)$ is the largest integer $k$ such that a maximum complement-free family $\mathcal{F}\subseteq\binom{[2n]}{n}$ satisfies $m(\mathcal{F})=2k+2$. Whether or not these functions are monotone is not immediate, and we are not currently able to determine this. Instead, we consider monotone counterparts of these functions. Let

\begin{equation*}
N^*(k)=\min\left\{n:\text{ there exists }\mathcal{F}\subseteq\binom{[2n]}{n}\text{ with }|\cF|=\frac{1}{2}\binom{2n}{n}\text{ and }2k+2\leq m(\mathcal{F})< \infty\right\},
\end{equation*} and
\begin{equation*}
K^*(n)=\max\left\{k:\text{ there exists }m\leq n\text{ and }\mathcal{F}\subseteq\binom{[2m]}{m}\text{ with }|\cF|=\frac{1}{2}\binom{2m}{m}\text{ and }m(\mathcal{F})=2k+2\right\}.
\end{equation*}

  Notice that from the definitions, $N^*(k)\leq N(k)$ and $K^*(n)\geq K(n)$ for all $n,k\geq1$, and
\[
N^*(k) = \min\{ N(j) : j \geq k \}\qquad\text{ and }\qquad K^*(n) = \max\{ K(m) : m \leq n \}.
\] 

Further, $N^*$ and $K^*$ are inverse of each other in the following sense.
\begin{proposition}\label{prop:ineq}
    For $n,k\geq1$, $N^*(k)\leq n$ if and only if $K^*(n)\geq k$.
\end{proposition}
\begin{proof}
    Suppose that $N^*(k)\leq n$. Then, by definition, there exists $j\geq k$ and $m\leq n$ with $N(j)=m$. Also by definition, since $N(j)=m$, then $K(m)\geq j\geq k$. Thus, $K^*(n)\geq K(m)\geq k$.

The converse is proved in a similar fashion. Indeed, if $K^*(n)\geq k$, then there exists $m\leq n$ with $j=K(m)= K^*(n)\geq k$. So by definition, $N^*(k)\leq N(j)\leq m\leq n$.
\end{proof}

In particular, it follows that 
\[
N^*(k)=\min\{n:K^*(n)\geq k\}\qquad\text{ and }\qquad K^*(n)=\max\{k:N^*(k)\leq n\}.
\]

We now provide two results relating to the asymptotics of $N,N^*,K,$ and $K^*$. The first result is a lower bound for $N(k), N^*(k)$ and an upper bound for $K(n), K^*(n)$.
 
\begin{proposition} The following statements are true.
\begin{enumerate}
    \item[(i)] $N(k)$ and $N^*(k)$ are $\Omega(\log k)$ as $k\to\infty$.
    \item[(ii)] $K(n)$ and $K^*(n)$ are $\displaystyle{O\left(\frac{4^n}{\sqrt{n}}\right)}$ as $n\to\infty$.
\end{enumerate}
\end{proposition}

\begin{proof}
We first prove (i). Let $n=N(k)$. By the definition of $N(k)$, there
exists a family $\mathcal F\subseteq\binom{[2n]}{n}$ such that $|\mathcal{F}|=\frac{1}{2}\binom{2n}{n}$ and $m(\mathcal{F})=2k+2$. Since a balanced subfamily of size $m(\mathcal{F})$ is contained in $\mathcal{F}$, we have
\[
2k+2=m(\mathcal{F})\leq|\mathcal{F}|=\frac12\binom{2n}{n}.
\]
Moreover, $\binom{2n}{n}\leq 2^{2n}=4^n$, since $\binom{2n}{n}$ is one of the terms in $\sum_{r=0}^{2n}\binom{2n}{r}=2^{2n}$. As a consequence, $2k+2\leq\frac{1}{2}4^n$, and hence $4k+4\leq4^n$. Taking logarithms, we get $n\geq 1+\log_4(k+1)$. Therefore, $N(k)=\Omega(\log k)$.

The same argument applies directly to $N^*(k)$. Let $n=N^*(k)$. By definition, there exists $\mathcal{F}\subseteq\binom{[2n]}{n}$ such that $|\mathcal{F}|=\frac{1}{2}\binom{2n}{n}$ and $2k+2\leq m(\mathcal F)<\infty$. Thus
\[
2k+2 \leq m(\mathcal{F}) \leq|\mathcal {F}| =\frac12\binom{2n}{n}\leq\frac{1}{2}4^n.
\]
Again, $n\geq1+\log_4(k+1)$, and therefore $N^*(k)=\Omega(\log k)$.

We now prove (ii). For sufficiently large $n$, the definition of $K(n)$ gives the existence of a family $\mathcal F\subseteq\binom{[2n]}{n}$ such that $|\mathcal{F}|=\frac{1}{2}\binom{2n}{n}$ and $m(\mathcal{F})=2K(n)+2$. Hence
\[
2K(n)+2 =m(\mathcal{F}) \leq|\mathcal{F}|=\frac{1}{2}\binom{2n}{n}.
\]
Stirling's approximation gives $\binom{2n}{n}\sim\frac{4^n}{\sqrt{\pi n}}$ as $n\to\infty$. In particular, there exist constants $C>0$ and $n_0$ such that $\binom{2n}{n}\leq C\frac{4^n}{\sqrt n}$ for every $n\geq n_0$. Therefore, $2K(n)+2 \leq\frac{C}{2}\frac{4^n}{\sqrt{n}}$, and consequently $K(n)=O\left(\frac{4^n}{\sqrt n}\right)$.

Finally, consider $K^*(n)$. By definition, there is some $m\leq n$ such that $K^*(n)=K(m)$. Hence $2K^*(n)+2 =2K(m)+2 \leq\frac{1}{2}\binom{2m}{m}$. Since the central binomial coefficients are increasing in $m$, we have that $\binom{2m}{m}\leq\binom{2n}{n}$. It follows that $2K^*(n)+2\leq\frac{1}{2}\binom{2n}{n}$. Applying Stirling's formula as above gives $K^*(n)=O\left(\frac{4^n}{\sqrt n}\right)$.
\end{proof}

The last result in this section is an upper bound for $N, N^*$ and a lower bound for $K, K^*$. 

\begin{proposition}
The following statements are true.
\begin{enumerate}
    \item[(i)] $N(k)$ and $N^*(k)$ are $O(k)$ as $k\to\infty$.
    \item[(ii)] $K(n)$ and $K^*(n)$ are $\Omega(n)$ as $n\to\infty$.
\end{enumerate}    
\end{proposition}
\begin{proof}
To prove (i), notice that for $k\geq3$, the game from Proposition~\ref{prop:constructions}(i) has $2k+3$ players. Since our construction adds an extra player $\st$, it follows that $N(k)\leq 2k+4$. Since $N^*(k)\leq N(k)$, it follows that both $N(k)$ and $N^*(k)$ are $O(k)$ as $k\to\infty$.

To prove (ii), recall that Proposition~\ref{prop:ineq} gives $N^*(k)\leq n$ if and only if $K^*(n)\geq k$. Since $N^*(k)\leq 2k+4$, it follows that $K^*(n)\geq \lfloor (n-4)/2\rfloor$. Therefore $K^*(n)$ is $\Omega(n)$ as $n\to\infty.$ 

We are left with proving that $K(n)\geq cn$ for some constant $c>0$ and all sufficiently large $n$.  For each $k\geq 3$, the
Gvozdeva and Slinko construction gives a simple game $G=(P,\mathcal W)$ on $2k+3$ players that is $k$-trade robust but not $(k+1)$-trade robust.  Thus there is a $(k+1)$-trade
\[
X_1,\ldots,X_{k+1} \longleftrightarrow Y_1,\ldots,Y_{k+1}.
\]
After adjoining the point $\st$, our self-dual construction is defined on $S=P\cup\{\st\}$ with $|S|=2k+4$. Taking the disjoint copy $S'=\{s':s\in S\}$ and applying the middle-layer construction of Theorem~\ref{thm:main} gives a family in $\binom{S\sqcup S'}{2k+4}$. Hence, $K(2k+4)\geq k$. It remains to deal with the values of $n$ not obtained directly in this way.  Since the values of the form $2k+4$ cover all sufficiently large even integers, it is enough to show that the construction at $n=2k+4$ can be extended to $n=2k+5$. The Padding Lemma establishes this. It follows that for every sufficiently large $n$, $K(n)\geq \left\lfloor\frac{n-4}{2}\right\rfloor$, and hence $K(n)=\Omega(n)$ as $n\to\infty$.
\end{proof}

\section{Final remarks} 
We proved that $N^*(k)$ and $N(k)$ are $\Omega(\log k)$ and $O(k)$ as $k\to\infty$, which leaves an exponential gap. For $K(n)$, the situation is similar and we proved that $K^*(n)$ and $K(n)$ are $\Omega(n)$ and $O(4^n/\sqrt{n})$. To formulate a conjecture regarding the growth rate of these functions, we computed attainable values of $k$ for small fixed values of $n$. 

For small values of $n$, computer computations allow us to conclude the following results. 

\begin{proposition}\label{prop:small-values}
We have
\[
N(1)=3, N(2)=5,\text{ and }N(k)\geq5\text{ for every }k\geq2.
\]
\end{proposition}

For $n\leq2$, a self-dual selector of $\binom{[2n]}{n}$ has at most three members. Since every balanced subfamily has even cardinality and a balanced subfamily of size $2$ is a complementary pair, such a family contains no balanced subfamily.

Moreover, for $n=3$, we exhaustively enumerated all $2^{\frac12\binom63}=2^{10}=1024$ self-dual selectors. Of these, $692$ have smallest balanced subfamily of size $4$, while the remaining $332$ contain no balanced subfamily. In particular, no finite value of $m(\mathcal{F})$ greater than $4$ occurs when $n=3$.

Unfortunately, for $n=4$, exhaustive enumeration of all the $2^{35}$ self-dual selectors is impractical. Instead, we formulated as a $0/1$ feasibility problem the existence of a self-dual selector $\mathcal{F}$ satisfying $6\leq m(\mathcal{F})<\infty$.

The model simultaneously allows every possible even size from $6$ to $34$ for the first balanced subfamily and excludes all balanced subfamilies of size $4$. Google OR-Tools CP-SAT\footnote{\url{https://developers.google.com/optimization/cp/cp_solver}} proved this system infeasible. Thus, for $n=4$ every self-dual selector satisfies either $m(\mathcal{F})=4$ or $m(\mathcal{F})=\infty$. Thus $N(k)\geq5$ for every $k\geq2$.

Finally, for $n=5$ we found a self-dual selector $\mathcal{F}\subseteq\binom{[2n]}{n}$ with $m(\mathcal{F})=6$. Hence $N(2)=5$.

For completeness, we briefly describe the $n=4$ feasibility model. There are $70$ sets in $\binom{[8]}4$, thus forming \(35\) complementary pairs. For every $A\in\binom{[2n]}{n}$ let $f_A\in\{0,1\}$ with $f_A=1\iff A\in\mathcal F$ and $y_A\in\{0,1\}$ with $y_A=1\iff A\in\mathcal S$, and $s$ be a non-negative integer. Let us write $\bar A=[2n]\setminus A$, and let $\mathcal{B}_4$ be the set of balanced $4$-families containing no complementary pair (there are $1260$ such balanced $4$-families). The complete set of constraints used is listed below. 
\begin{align*}
  f_A+f_{\bar A} &= 1 && \text{for all } A
     && \text{($\mathcal F$ is a self-dual selector)}\\
  y_A &\le f_A && \text{for all } A
     && \text{($\mathcal S\subseteq\mathcal F$)}\\
  \textstyle\sum_{A\ni i} y_A &= s && \text{for all } i\in[2n]
     && \text{($\mathcal S$ balanced, $|\mathcal S|=2s$)}\\
  s &\ge 3 && && \text{($|\mathcal S|\ge 6$)}\\
  \textstyle\sum_{A\in B} f_A &\le 3 && \text{for all } B\in\mathcal B_4
     && \text{(no balanced $4$-subfamily)}\\
  y_{[n]} &= 1 && && \text{(wlog)}
\end{align*}

We additionally fix one member of the selected balanced subfamily by requiring $y_{\{1,2,3,4\}}=1$. This causes no loss of generality, since the natural action of $S_8$ is transitive on $\binom{[8]}4$. The resulting system is infeasible. Hence there is no self-dual selector $\mathcal{F}\subseteq\binom{[8]}4$ satisfying $6\leq m(\mathcal{F})<\infty$.

We summarize the computational results obtained cin Table~\ref{tab:results}. 

\begin{table}[ht]
\centering
\begin{tabular}{|c|c|c|}
\hline
$n$ & $k$ (so $m(\mathcal{F})=2k+2$) & result \\
\hline
$1,2$ & all $k\ge1$ & infeasible  \\
\hline
$3$ & $1$ & feasible \\
\hline
$3$ & $k\ge2$ & infeasible \\
\hline
$4$ & $1$ & feasible  \\
\hline
$4$ & $k\ge2$ & infeasible  \\
\hline
$5$ & $1$ & feasible \\
\hline
$5$ & $2$ & feasible  \\
\hline
$5$ & $k\ge3$ & unknown  \\
\hline
\end{tabular}
\caption{Pairs $(n,k)$ tested. ``Feasible'' means that there is a self-dual selector $\mathcal{F}\subseteq\binom{[2n]}{n}$ (so $|\mathcal{F}|=\frac{1}{2}\binom{2n}{n}$) with $m(\mathcal{F})=2k+2$.}
\label{tab:results}
\end{table}

These values are too small to distinguish between linear and logarithmic growth. We nevertheless suspect that the lower bound is closer to the truth, at least along an infinite sequence of values of $k$.  This leads to the following question.

\medskip

\textbf{Open Question 1: } Is $\displaystyle{\liminf_{k\to\infty} N(k)/\log k<\infty\;?}$ Equivalently, is there a constant $c>1$ such that $K(n)\ge c^{\,n}$ for infinitely many $n$?
\medskip 

As stated, the numbers we were able to compute are too small to even conjecture an answer to this question. 

\section*{Acknowledgments} The author thanks Lawrence S. Moss for bringing the question to his attention, for sharing the then unpublished manuscript~\cite{MossPedersen}, and for comments provided on earlier drafts of this note. All errors that remain are the author's.

\end{document}